\documentclass[12pt,reqno,oneside]{amsart}
\usepackage{amsmath}
\usepackage{amssymb}
\usepackage{amsthm}
\usepackage{graphics}
\usepackage{eucal}
\usepackage{mathrsfs}
\usepackage{color}
\usepackage[pdftex,bookmarks,colorlinks,breaklinks]{hyperref}
\usepackage{mathtools}
\newtheorem{theorem}{Theorem}[section]
\newtheorem{definition}{Definition}[section]
\newtheorem{proposition}{Proposition}[section]
\newtheorem{lemma}{Lemma}[section]
\newtheorem{corollary}{Corollary}[section]

\numberwithin{equation}{section}

\newcommand{\diag}{\operatorname{diag}}

\newcommand{\GL}{\operatorname{GL}}

\newcommand{\supp}{\operatorname{supp}}
\newcommand{\cov}{\operatorname{cov}}

\newcommand{\bt}{\mathbf{t}}
\newcommand{\bq}{\mathbf{q}}
\newcommand{\bx}{\mathbf{x}}

\newcommand{\cL}{\mathcal{L}}

\newcommand{\bw}{\mathbf{w}}
\newcommand{\f}{\mathbf{f}}
\newcommand{\rk}{\operatorname{rank}}

\newcommand{\bc}{\mathbf{c}}

\newcommand{\by}{\mathbf{y}}

\newcommand{\al}{\alpha}

\newcommand{\Q}{\mathbb{Q}}
\newcommand{\Z}{\mathbb{Z}}
\newcommand{\R}{\mathbb{R}}
\newcommand{\N}{\mathbb{N}}

\newcommand{\inv}{^{\text{-}1}}

\usepackage{mathtools}

\begin{document}

\title[$p$-adic Diophantine approximation]{Multiplicative $p$-adic metric Diophantine approximation on manifolds and dichotomy of exponents}

\author{Shreyasi Datta}
\address{\textbf{Shreyasi Datta} \\
School of Mathematics,
Tata Institute of Fundamental Research, Mumbai, India 400005}
\email{shreya@math.tifr.res.in}

\author{Anish Ghosh}
\address{\textbf{Anish Ghosh} \\
School of Mathematics,
Tata Institute of Fundamental Research, Mumbai, India 400005}
\email{ghosh@math.tifr.res.in}
\date{}

\thanks{A.\ G.\ gratefully acknowledges support from a grant from the Indo-French Centre for the Promotion of Advanced Research; a Department of Science and Technology, Government of India Swarnajayanti fellowship and a MATRICS grant from the Science and Engineering Research Board.}

\date{}

\begin{abstract}
In this paper we study $p$-adic Diophantine approximation on manifolds, specifically multiplicative Diophantine approximation on affine subspaces and a Diophantine dichotomy for analytic $p$-adic manifolds. 
\end{abstract}

\maketitle

\section{Introduction}\label{sec:intro}
This paper is concerned with $p$-adic Diophantine approximation on manifolds, specifically with multiplicative Diophantine approximation and a Diophantine dichotomy in the $p$-adic setting.  In an earlier paper \cite{DGpart1}, we introduced a new $p$-adic Diophantine exponent and answered questions of Kleinbock  and Kleinbock-Tomanov concerning $p$-adic Diophantine approximation on affine subspaces. The Diophantine exponent we introduced is better suited to homogeneous dynamics. In this paper, we continue our study by establishing multiplicative versions of our results and also establishing a Diophantine dichotomy for $p$-adic analytic manifolds.   
For $\bq = (q_1, \dots, q_n) \in \Z^n$ and $q_0 \in \Z$, set  $\tilde{\bq} := (q_0, q_1, \dots, q_n)$. Following Kleinbock and Tomanov, \cite{KT}, we define the Diophantine exponent $w(\by)$ of $\by \in \Q_p^{n}$ to be the supremum of $v > 0$ such that there are infinitely many $\tilde{\bq} \in \Z^n$ satisfying
\begin{equation}\label{def:exp}
|q_0 + \bq\cdot \by_p  \leq \|\tilde{\bq}\|_\infty^{-v}.
\end{equation}
In view of Dirichlet's theorem (\cite{KT} \S 11.2), $w(\by) \geq n+1$ for every $\by \in \Q_p^{n},$ with equality for Haar almost every $\by$ by the Borel-Cantelli lemma. We will also need the following definition from \cite{DGpart1}.

\begin{definition}{\textbf{v-$\Z[1/p]$ approximable vectors:}} $\by\in\Q_p^n$ is $v-\Z[1/p]$-approximable if there exist $\tilde{\bq}$ with unbounded $||\bq||_p||\tilde{\bq}||_\infty$ such that 
		\begin{equation}\label{Z_S}
		|\bq.\by+q_0|_p<\frac{1}{(||\bq||_p||\tilde{\bq}||_\infty)^v||\tilde{\bq}||_\infty},
		\end{equation}
		 where we recall that $\tilde{\bq}=(q_0,\bq)\in\Z[1/p]^{n+1}$.\end{definition}
\noindent
We will denote $v-\Z[1/p]$-approximable points by $\mathcal{W}_v^{p}$ and also define 
\begin{equation}\label{def:papp}
w_{p}(\by):=\sup\{v \text{ appearing in } (\ref{Z_S}) \}.
\end{equation}

By Proposition $3.1$ from \cite{DGpart1}, we have that for any $\by\in\Q_p^n$ we have 
\begin{equation}\label{relation}
w_{p}(\by)=w(\by)+1.
\end{equation}
In \cite{DGpart1}, we studied the exponents $w_p$ and $w$ in detail and proved an inheritance result for the exponent $w$ when restricted to submanifolds. This result was proved using the dynamical technique of Kleinbock and Margulis, and the exponent $w_p$ played a key role.\\

 In \cite{Kleinbock-dichotomy}, D. Kleinbock established a remarkable \emph{dichotomy} with regards to certain Diophantine properties. In particular, he proved that for connected, analytic manifolds, having one not very well approximable point implies that almost every point is not very well approximable. Our first result addresses this in the $p$-adic context.

 \begin{theorem}\label{theorem-dicho}
		Suppose $\mathcal{M}$ is an analytic manifold of $\Q_p^n$. Let $v\geq n+1$ and suppose $w(\by_0)\leq v$ for some $\by_0\in\mathcal{M}$ then for almost every $\by$  around a neighbourhood of $\by_0$, $w(\by)\leq v$. 
	\end{theorem}
 
We now turn our attention to multiplicative Diophantine approximation, see \cite{Bugeaud} for a survey of this topic which has several parallels with, as well as striking differences from, Diophantine approximation with the usual norm.  For $\bq = (q_1, \dots, q_n) \in \Z^n$ and $q_0 \in \Z$, set  $\tilde{\bq} := (q_0, q_1, \dots, q_n)$.  We define $$\vert q\vert_{+}=
	\left\{\begin{array}{ll} ~ \vert q\vert_\infty  &\text{ if } q\neq 0 \\
	\\ ~1&~ \text{otherwise}
	\end{array}\right .$$ for $q\in\mathcal{D}$. For $q\in\Z$ this definition matches with the classical one. Further define $$\Pi_{+}(\tilde{\bq})=\prod_{i=0,\cdots,n} \vert q_i\vert_{+}.$$

Say  that $\by \in \Q^{n}_p$ is very well multiplicatively approximable (VWMA) if, for some $\varepsilon > 0$ there are infinitely many solutions $\tilde{\bq} = (q_0, q_1, \dots, q_n) \in \Z^{n+1}$ to
\begin{equation} 
|q_0 + \bq\cdot \by|_p  \leq \Pi_{+}(\tilde{\bq})^{-(1+\varepsilon)}. 
\end{equation}




 Following Kleinbock \cite{Kleinbock-exponent}, we say that a differentiable map $f : U \to \Q^{n}_{p}$, where $U$ is an open subset of $\Q^{d}_{p}$, is nondegenerate in an affine subspace $\mathcal{L}$ of $\Q^{n}_p$ at $x \in U$ if $f(U) \subset \mathcal{L}$ and the span of all the partial derivatives of $f$ at $x$ up to some order coincides with the linear part of $\mathcal{L}$. If $\mathcal{M}$ is a $d$-dimensional submanifold of $\mathcal{L}$, we will say that $\mathcal{M}$ is nondegenerate in $\mathcal{L}$ at $y \in \mathcal{M}$ if there exists a diffeomorphism $f$ between an open subset $U$ of $\Q^{d}_{p}$ and a neighbourhood of $y$ in $\mathcal{M}$ is nondegenerate in $\mathcal{L}$ at $f^{-1}(y)$. We also denote the Haar measure on $\Q_p^d$ by $\lambda$. Finally, we will say that $f:U\to \mathcal{L}$ (resp., $\mathcal{M} \subset \mathcal{L}$) is nondegenerate in $\mathcal{L}$ if it is nondegenerate in $\mathcal{L}$ at $\lambda$-a.e. point of $U$ (resp., of $\mathcal{M}$, in the sense of the smooth measure class on $\mathcal{M}$).

\noindent A manifold is  called strongly extremal if almost every point with respect to the Lebesgue measure is not very well multiplicatively approximable.  Our second result resolves the conjecture by Kleinbock and Tomanov in \cite{KT} in the multiplicative case.


\begin{theorem}\label{thm:mult}
	Let $\mathcal{L}$ be an affine subspace of $\Q_p^n$
and let $\f : \Q_p^d\mapsto \mathcal{L}$  be a $C^k$ map which is nondegenerate in $\mathcal{L}$ at $\lambda$ a.e. point. Suppose
that the volume measure
 $\lambda$ on $\mathcal{L}$ 
 is strongly extremal, then so is $\f_*\lambda$.                                                                          
\end{theorem}
Our third result shows a dichotomy for the multiplicative case as follows.
 \begin{theorem}
 		For any analytic manifold of $\Q_p^n$ if one point is not very well multiplicatively approximable then the set of not very well multiplicatively approximable points has positive measure.
 \end{theorem}
 
The approach used in this paper uses homogeneous dynamics as introduced by Kleinbock and Margulis in their important paper \cite{KM}. Diophantine approximation on affine subspaces has seen several developments recently, we refer the reader to \cite{BBDD, Kleinbock-extremal, Kleinbock-exponent, G1, G-thesis, G-div, G-mult, G-monat, GM, GG} and \cite{G-handbook} for a survey. Likewise, Diophantine approximation on manifolds in the $p$-adic setting has been studied in \cite{KT} and subsequently in \cite{BK, MoS1, MoS2, DG, DGpart1}.


\section{Quantitative nondivergence}
We will use quantitative nondivergence estimates for certain flows on homogeneous spaces. This estimate has its origin in the influential work of Kleinbock and Margulis \cite{KM}. We refer the reader to the recent survey \cite{BKsurvey} for instances of the ubiquity of quantitative nondivergence. In the context of the present paper, the most relevant developments are an $S$-adic version of quantitative nondivergence developed by Kleinbock and Tomanov \cite{KT} and an improved estimate, crucial to Diophantine applications, developed by Kleinbock in \cite{Kleinbock-exponent}. A $p$-adic version of this estimate was used in our earlier work \cite{DGpart1} and will also play a central role in the present paper. We recall some notation and definitions and state Theorem 5.3 from \cite{DGpart1}, a $p$-adic version of D. Kleinbock's improved quantitative nondivergence theorem from \cite{Kleinbock-exponent}. This theorem more or less follows from \cite{Kleinbock-exponent} by adapting the necessary changes from \cite{KT}. For completeness, a proof is provided in \cite{DGpart1}.

We need some definitions and notation and follow \cite{KT} in our exposition.
A metric space $X$ is called \emph{Besicovitch} \cite{KT} if there exists a constant $N_X$ such that the following holds: for any bounded subset $A$ of $X$ and for any family $\mathcal{B}$ of nonempty open balls in $X$ such that
$$ \forall x \in A \text{ is a center of some ball of } B,$$
there is a finite or countable subfamily $\{B_i\}$ of $B$ with
$$ 1_A \leq \sum_{i}1_{B_i} \leq N_X. $$

\noindent  We now define $D$-Federer measures following \cite{KLW}. Let $\mu$ be a Radon measure on $X$, and $U$ an open subset of $X$ with $\mu(U) > 0$. We  say that $\mu$ is $D$-Federer on $U$ if
$$ \sup_{\substack{x \in \supp \mu, r > 0\\ B(x, 3r) \subset U}} \frac{\mu(B(x, 3r))}{\mu(x,r)} < D.$$
Finally, we say that $\mu$ as above is Federer if for $\mu$-a.e. $x \in X$ there exists a neighbourhood $U$ of $x$ and $D > 0$ such that $\mu$ is $D$-Federer on $U$. We refer the reader to \cite{KLW, KT} for examples of Federer measures.\\
Following, \cite{Kleinbock-exponent}, for a subset $M$ of $\Q_{p}^n$, define its affine span $\langle M\rangle_a$ to be the intersection of all affine subspaces of $\Q_{p}^n$ containing $M$. Let $X$ be a metric space, $\mu$ a Borel measure on $X$, $\cL$ an affine subspace of  $\Q_{p}^n$ and $f$ a map from $X$ into $\cL$. Say that $(f, \mu)$ is nonplanar in $\cL$ if
$$\cL = \langle f (B \cap \supp \mu)\rangle_a \forall \text{ nonempty open } B \text{ with } \mu(B) > 0.$$

For a subset $U$ of $X$ and $C, \alpha > 0$, say that a Borel measurable function $f : U \to \Q_p$ is $(C, \alpha)$-good on $U$ with respect to $\mu$ if for any open ball $B \subset U$ centered in $\supp \mu$ and $\varepsilon > 0$ one has
\begin{equation}\label{gooddef}
\mu \left(\{ x \in B \big| |f(x)| < \varepsilon \} \right) \leq
C\left(\displaystyle \frac{\varepsilon}{\sup_{x \in
B}|f(x)|}\right)^{\al}|B|.
\end{equation}

Where $\|f\|_{\mu, B} = \sup \{c : \mu(\{x \in B : |f(x)| > c\}) > 0\}$.\\

Let $S = \{\infty, p\}$,  $\Q_S:= \Q_p \times \R$,  $\mathcal{D} := \mathbb{Z}[1/p]$ and
$$ \mathcal{P}(\mathcal{D}, m):= \text{ the set of all nonzero primitive submodules of } \mathcal{D}^{m}.$$

A vector $ \bx$ of $\Q^{m}_S$ will be denoted as $\bx = (x^{(\infty)}, x^{(p)})$, where $x^{(v)} = (x^{(v)}_1, . . . , x^{(v)}_m) \in \Q_v^m$. The norm $\|\bx\|$ and the content $c(\bx)$ of $\bx$ are defined to be the maximum (resp., the product) of all the numbers $\|(x^{(v)})\|_v$.

For a discrete $\mathcal{D}$-submodule $\Lambda$ of $\Q^{m}_S$, we set
$$\delta(\Lambda):= \min\{c(\bx) : \bx \in \Lambda\backslash \{0\}\},$$
The covolume, cov() of a lattice used below is defined as in \S 8.3 of \cite{KT}.

We recall Theorem $5.3$ from \cite{DGpart1}.

\begin{theorem}\label{QND2}
		Let $X$ be a 
		be a Besicovitch metric space,
		$\mu$ a  uniformly Federer measure  on $X$, and let
		$S$
		be as above.  For  $m\in \N$, let  a
		ball $B = B(x_0,r_0)\subset X$ and a continuous map $h:\tilde B
		\to \GL(m,\Q_S)$ be given, where $\tilde B$ stands for
		$B(x_0,3^mr_0)$.
		Now suppose that for some
		$C,\alpha > 0$ and $0 < \rho  < 1$  one has
		
		{\label{rm1}\rm(i)} for every $\,\Delta\in \mathcal P(\mathcal{D},m)$, the function
		$\cov\big(h(\cdot)\Delta\big)$ is $(C,\alpha)$ good \ on $\tilde B$  with
		respect to $\mu$;
		
		{\label{rm2}\rm(ii)}  for every $\,\Delta\in \mathcal P(\mathcal{D},m)$,
		$\sup_{x\in B\cap \supp\mu}\cov\big(h(x)\Delta\big) \ge \rho^{rk(\Delta)}$.
		
		Then 
		for any  positive $ \varepsilon\le
		\rho$ one has
		$$
		\mu\left(\big\{x\in B\bigm| \delta \big(h(x)\mathcal{D}^m\big) < \varepsilon
		\big\}\right)\le mC \big(N_{X}D_{\mu}^2\big)^m
		\left(\frac\varepsilon \rho \right)^\alpha \mu(B)\
		$$
	\end{theorem}

\section{Dichotomy: Proof of Theorem \ref{theorem-dicho}}\label{aevsno}

In this section, we address $p$-adic versions of D. Kleinbock's paper \cite{Kleinbock-dichotomy} where he proved that analytic manifolds possess a remarkable dichotomy with regard to certain Diophantine properties, see also \cite{CaoYu} and \cite{Mosh}. 
We begin with
\begin{lemma}\label{minkowski}
For any $\Delta\in \mathcal{P}(\mathcal{D}, n+1)$ and $g\in \GL(n+1, \Q_S)$ we have $$
\delta(g\mathcal{D}^{n+1})\leq \cov(g\Delta)^{\frac{1}{rk(\Delta)}}.$$
\end{lemma}
\begin{proof}
	Note that $g\Delta$ is a lattice in $\Q_S g\Delta$. Set $j = \rk(\Delta)$. Now consider the ball \begin{align*}
	D&=D_\infty\times D_p\\
	&=\left\{ x^{(\infty)}\in\R^{n+1}\left| ||x^{(\infty)}||_\infty\leq (\cov(g\Delta))^{\frac{1}{j}}\right.\right\}\times
	\left\{ x^{(p)}\in\Q_p^{n+1}\left| \begin{aligned}
	&|x^{(p)}_1|_p\leq 1\\
	&|x^{(p)}_2|_p\leq 1\\
	&\vdots \\
	&|x^{(p)}_{n+1}|_p\leq 1
	\end{aligned}\right.\
	\right	\}.
	\end{align*} Denote $ D\cap \Q_Sg\Delta=D_1$, a ball in $\Q_Sg\Delta$.  The normalized Haar measure on $\Q_S\Delta_1$ where $\Delta_1=g\Delta$ is defined as
	$\lambda_S(D_1)=\mu_S(\pi\inv(D_1))$ where $\mu_S$ is the Haar measure on $\Q_S^j$ and $\pi:\Q_S^j\mapsto\Q_S\Delta_1=\Q_Sv_1+\cdots+\Q_Sv_j$ and $v_1,v_2,\cdots,v_j$ are taken such that $v_1^{\infty},\cdots,v_j^{\infty}$ form an orthonormal basis of $(\Q_S\Delta_1)_\infty$ and $(\Q_S\Delta_1)_p\cap \Z_p^{n+1}=\Z_pv_1^{(p)}+\cdots+\Z_pv_j^{(p)}$. So $$
	\lambda_S(D_1)\geq 2^j \cov(g\Delta).$$ Hence by Minkowski's theorem there exists $g\gamma\in g\Delta$ i.e. $\gamma\in\mathcal{D}^{n+1}$ such that 
	$$\begin{aligned}&
	\Vert (g\gamma)^{(\infty)}\Vert_\infty \leq  cov(g\Delta)^{\frac{1}{j}}\\
	&\Vert (g\gamma)^{(p)}\Vert_p\leq 1.
\end{aligned}
$$ Therefore $c(g\gamma)\leq \cov(g\Delta)^{\frac{1}{rk(\Delta)}}\implies \delta(g\mathcal{D}^{n+1})\leq \cov(g\Delta)^{\frac{1}{rk(\Delta)}}. $
	\end{proof}

\begin{proposition}
	Let $U$ be an open subset of $\Q_p^d$ and let $\mathcal{F}$ be a finite-dimensional space of analytic $\Q_p$ valued functions on $U$. Then for any $\bx\in U$ there exists $C,\alpha>0$ and a neighbourhood $W$ of $\bx$ such that every element of $\mathcal{F}$ is $(C, \alpha)$-good on $W$. 
\end{proposition}
\begin{proof}
	 Without loss of generality we may assume that  $\mathcal{F}$ contains constant functions. Let $1,f_1,\cdots,f_n$ be a basis of $\mathcal{F}$. Then $\f=(f_1,f_2,\cdots,f_n)$ is nonplanar (cf. \cite{KT}). For analytic functions nondegeneracy is equivalent to nonplanarity. Therefore, the conclusion follows from Proposition 4.2 of \cite{KT}.
\end{proof}

The Corollary below now follows from the expression of $\cov(g_tu_\f(\bx)\Delta)$ in (6.12) from \cite{DGpart1} as a 
maximum of norms of linear combinations of $f_i$'s.
\begin{corollary}\label{good_neighbourhood}
	Let $U$ be an open subset of $\Q_p^d$, and let $\f: U\mapsto \Q_p^n$ be an analytic map. Then for any $\bx_0\in U$ there exists $C, \alpha>0$ and a neighbourhood $W$ of $\bx_0$ contained in $U$ such that for any $\Delta\in\mathcal{P}(\mathcal{D},n+1)$ and $t\in N$ the functions $\bx\mapsto \cov(g_tu_{\f(\bx)}\Delta)$ are $(C,\alpha)$-good on $W$. 
\end{corollary}

Define $$\gamma(\by):=\sup\{c\geq 0\left | \delta(g_tu_{\by}\mathcal{D}^{n+1})\leq p^{-ct} \text{ for infinitely many } t\in\N \right \}$$ for $\by\in \Q_p^n$.
From Proposition 4.1 and Lemma 6.1 of \cite{DGpart1}  we can conclude that
\begin{equation}\label{gamma_w}
w_{p}(\by)=\frac {n(1+\gamma(\by))+\gamma(\by)}{1-(n+1)\gamma(\by)}.
\end{equation}
We are now ready for
\begin{theorem}\label{gamma}
	Suppose $\f:U \mapsto \Q_p^n $ is an analytic map and $U$ is an open subset of $\Q_p^d$. Denote by $\lambda$ the Haar measure on $\Q_p^d$. Let $\gamma\geq0$ and $\bx_0\in U$ be such that $\gamma(f(\bx_0))\leq \gamma$, then $\lambda$-almost every $\bx$ in a neighbourhood of $\bx_0$, we have $\gamma(\f(\bx))\leq \gamma$. 
\end{theorem}
\begin{proof}
	Let $\gamma\geq 0$ and $\bx_0 \in U$ be such that $\gamma(\f(\bx_0))\leq \gamma$. We consider the sets $U_1:=\left\{\bx\in U ~| ~\gamma(\f(\bx))\leq \gamma\}\right.$  and 
	$$ U_2:= \{ \bx\in U~|~ \mu(B\setminus U_1)=0 \text{ for some neighbourhood } B \text{ of } \bx\}.$$
	 We claim that $U_2=\bar U_1\cap U$. So if $\bx_0\in U_1,$ then $\bx_0\in U_2,$ which implies that $U_2$ is a nonempty open set and the conclusion of the theorem follows.\\
	 \noindent
	 Now take a point $\bx_1\in \bar U_1\cap U$ and a ball $B$ of $\bx_1$ such that $\tilde{B} :=3^{d-1}B$ is inside a neighbourhood appearing in Corollary \ref{good_neighbourhood}. We want to apply Theorem \ref{QND2} to the function $\bx\mapsto g_tu_{\f(\bx)}$. The first  condition (\rm{i}) of the Theorem is satisfied by Corollary \ref{good_neighbourhood}. Since $\bx\in \bar U_1,$ there exists $\bx'\in B\cap U_1$ and this implies that $\gamma(\f(\bx'))\leq \gamma,$ which in turn implies that 
	 $$\forall \gamma'>\gamma, ~ \delta(g_tu_{\f(\bx')}\mathcal{D}^{n+1})\geq p^{-\gamma't} \text{ for all but finitely many } t\in\N.$$  
	 Now applying Lemma \ref{minkowski} we have that $$\cov(g_tu_{\f(\bx')}\Delta)^{\frac{1}{rk(\Delta)}}\geq p^{-\gamma't}$$ for all $\Delta\in\mathcal{P}(\mathcal{D}, n+1)$ and for all but finitely many $t\in \N$. Hence condition (\rm{ii}) of the Theorem is satisfied. Taking $\varepsilon=p^{-\gamma''t}$ where $\gamma''>\gamma'$ and applying Theorem \ref{QND2} we have $$
	 \lambda\left\{\bx\in B~| ~ \delta(g_tu_{\f(\bx)}\mathcal{D}^{n+1})< p^{-\gamma''t}\right\}\leq C'\left(\frac{p^{-\gamma''t}}{p^{-\gamma't}}\right)^\alpha\mu(B).$$
	 By the Borel-Cantelli lemma we immediately have that for $\mu$-a.e $\bx\in B$  we have that $$
	 \delta(g_tu_{\f(\bx)}\mathcal{D}^{n+1})\leq p^{-\gamma''t}$$ for infinitely many $t\in\N$. Hence by definition for $\lambda$-a.e $\bx\in B$,  $\gamma(\f(\bx))\leq \gamma$ as $\gamma'', \gamma'$ were taken arbitrary close to $\gamma$. This implies $\bx_1\in U_2$ giving $\bar U_1\cap U\subset U_2$ and clearly $U_2\subset \bar U_1$.
	 \end{proof}
 Now from formula \ref{gamma_w} we can conclude the following:
	 \begin{corollary}
	 	Suppose $\f:U \mapsto \Q_p^n $ is an analytic map and $U$ is an open subset of $\Q_p^d$. Let $v\geq n$ and $\bx_0\in U$ be such that $w_{p}(f(\bx_0))\leq v$ then for $\lambda$-almost every $\bx$ in a neighbourhood of $\bx_0$ we have $w_{p}(\f(\bx))\leq v$. 
	 \end{corollary}
 
Since one can parametrize  analytic manifolds by images of open neighbourhoods under analytic functions, we have the following Theorems.
 \begin{theorem}
 	Suppose $\mathcal{M}$ is an analytic manifold of $\Q_p^n$. Let $v\geq n$ and suppose $w_{p}(\by_0)\leq v$ for some $\by\in\mathcal{M}$ then for almost every $\by$ in a neighbourhood of $\by_0$, $ w_{p}(\by)\leq v$. 
 \end{theorem}
Therefore by Proposition 3.1 from \cite{DGpart1}, see (\ref{relation}) we also have,  
	\begin{theorem}
		Suppose $\mathcal{M}$ is an analytic manifold of $\Q_p^n$. Let $v\geq n+1$ and suppose $w(\by_0)\leq v$ for some $\by_0\in\mathcal{M}$ then for almost every $\by$ in a neighbourhood of $\by_0$, $\omega(\by)\leq v$. 
	\end{theorem}
	So if one point in an analytic $p$-adic manifold is not very well approximable, then the set of not very well approximable points has positive measure. Note that this phenomenon was already clear from 
 Theorem 6.2 of \cite{DGpart1} for the manifolds which were nondegenerate inside some affine subspace. In fact, in that case the set of not VWA would have measure full by existence of one not VWA vector. The theorems above constitute $p$-adic analogues of  Theorem $1.4$ (a) of \cite{Kleinbock-dichotomy}. We have not pursued part (b) of the Theorem in loc. cit. which has to do with singular vectors.	

\section{Multiplicative Diophantine approximation}\label{multi}

The multiplicative analogues of Sprind\v{z}huk's conjectures were formulated by Baker and settled by Kleinbock and Margulis in \cite{KM}. In \cite{Kleinbock-extremal}, D. Kleinbock proves his results for affine subspaces and their nondegenerate manifolds also in the multiplicative context. The setup is more subtle but the dynamical approach is powerful enough to deal with this, one replaces the one-parameter diagonal action with a multiparameter action. In \cite{G-mult}, the second named author proved a multiplicative version of a Khintchine type theorems for hyperplanes. Further in \cite{KT}, the authors established the $S$-adic Baker-Sprind\v{z}huk conjectures, namely they also considered the multiplicative case. In \S 6.3 of \cite{Kleinbock-exponent}, D. Kleinbock refers to the possibility of proving his improved exponent results also for subspaces, some of this was accomplished in \cite{Zhang1}.

\subsection{A Dynamical correspondence for $p$-adic VWMA vectors}
In this section we define $g_{\bt}\in \GL_{n+1}(\Q_p\times\R)$ such that 
$$g_{\bt}^p:=\diag(p^{-t}, 1,\cdots,1) \text{ and } g_{\bt}^\infty:=\diag(p^{-t_0}, p^{-t_1},\cdots,p^{-t_n})$$ 
where $\bt=(t_0,\cdots,t_n)$ and $t :=\sum_{i=0}^n t_i$.	\begin{lemma}\label{dani_vwma}
		A vector $\by\in\Q_p^n$ is very well multiplicatively approximable if and only if there exist unbounded $t>0$ with $\bt\in\Z_+^{n+1} $ such that $$\delta(g_{\bt}u_{\by}\mathcal{D}^{n+1})\leq p^{-\gamma t}$$ for some $0<\gamma<\frac{1}{n+1}$.
\end{lemma}
\begin{proof}
 Suppose $\delta(g_{\bt}u_{\by}\mathcal{D}^{n+1})\leq p^{-\gamma t}$, this implies that for unbounded $t>0$ and $\tilde\bq=(q_0,\bq)\in\mathcal{D}^{n+1}$, we have that
\begin{equation}\label{hyp}
\max(p^t\vert q_0+\bq\cdot{\by}\vert_p,\Vert\bq\Vert_p)\max_{i=0,\cdots,n} (p^{-t_i}\vert q_i\vert_\infty)\leq p^{-\gamma t},
\end{equation}
which implies that
\begin{equation}
\max(p^t\big\vert \Vert\bq\Vert_p q_0+\Vert\bq\Vert_p\bq\cdot{\by}\big\vert_p\Vert\bq\Vert_p,\Vert\bq\Vert_p)\max_{i=0,\cdots,n} (p^{-t_i}\vert q_i\vert_\infty)\leq p^{-\gamma t}.
\end{equation}
This implies that \begin{equation}\label{condition}
		\max(p^t\vert q'_0+\bq'\cdot{\by}\vert_p\max_{i=0,\cdots,n} (p^{-t_i}\vert q'_i\vert_\infty),\max_{i=0,\cdots,n} (p^{-t_i}\vert q'_i\vert_\infty))\leq p^{-\gamma t}
		\end{equation} 
		where $\tilde \bq'=\Vert\bq\Vert_p\tilde{\bq}$. Since $\gamma>0$ we have $\bq\neq 0$. Hence for $i=1,\cdots,n$ we have $\vert q'_i\vert_p=\frac{\vert q_i\vert_p}{\Vert\bq\Vert_p}\leq 1$ giving $q'_i\in\Z$ for such $i$.
		There exists $1\leq k\leq n+1$ such that $\#\{i~|~ t_i\geq \gamma t\}=k $, possibly considering a subsequence of $\bt$. From $\max_{i=0,\cdots,n}(p^{-t_i}\vert q'_i\vert_\infty)\leq p^{-\gamma t}$ we have that
		$$\vert q'_i\vert_\infty\leq p^{-\gamma t+t_i} \text{ and } 1\leq p^{-\gamma t+t_i}$$ 
		for $k$ no of $i$. Otherwise we have $\vert q'_i\vert_\infty\leq p^{-\gamma t+t_i}<1$.
		Therefore \begin{equation}\label{condition2}
		\begin{aligned}\Pi_{+}(\tilde \bq')\leq &\prod_{i \text{ s.t }  t_i\geq \gamma t} \vert q'_i\vert_{+}\\
		& \leq p^{-k\gamma t+\sum_{i \text{ s.t }  t_i\geq \gamma t} t_i}\\
		&\leq p^{(1-k\gamma)t}.      
		\end{aligned}\end{equation}
		Again possibly taking a subsequence, there exists $1\leq m\leq n+1$  such that $\#\{i~|~ q'_i\neq 0\}=m$. By the definition $q'_i \neq 0$ implies $\vert q_i\vert_{+}=\vert q_i\vert_{\infty}$. Now by (\ref{condition}) we have that  
		$$\vert q'_0+\bq'\cdot\by\vert_p \vert q'_i\vert_{\infty}\leq \frac{p^{-\gamma t}}{p^{t-t_i}}.$$ 
		Multiplying over all those $i$ such that $q'_i\neq 0$,
		$$\begin{aligned}
		\vert q'_0+\bq'\cdot \by\vert_p^m\cdot \Pi_{+}(\tilde{\bq}')&\leq \frac{p^{-m\gamma t}}{p^{mt-\sum_{i \text{ s.t }q'_i\neq 0}t_i}}\\ & \leq \frac{p^{-m\gamma t}}{p^{mt}}\cdot p^t\\
		&\leq p^{-(m\gamma+m-1)t}.
		\end{aligned}$$ Now from (\ref{condition2}) it follows that $$
		\begin{aligned}
		\vert q'_0+\bq'\cdot\by\vert_p^m&\leq \Pi_{+}(\tilde{\bq}')^{-(\frac{\gamma m+m-1}{1-k\gamma}+1)}\\
		& \leq \Pi_{+}(\tilde{\bq}')^{-\frac{m\gamma +m-k\gamma}{1-k\gamma}}.
		\end{aligned}$$
		Hence we have that 
		$$\vert q'_0+\bq'\cdot\by\vert_p\leq \Pi_{+}(\tilde{\bq}')^{-\frac{(m\gamma +m-k\gamma)}{m(1-k\gamma)}}.$$ 
		Since $\gamma<\frac{1}{n+1}<\frac{1}{k}$ we have $$\frac{(m\gamma +m-k\gamma)}{m(1-k\gamma)}-1=\frac{\gamma(m-k+mk)}{m(1-k\gamma)}=\varepsilon>0.$$
		 We now claim that the $\Pi_{+}(\tilde{\bq}')$ appearing here are unbounded. Note that $ \vert q_0+\bq\cdot \by\vert_p^m\cdot \Pi_{+}(\tilde{\bq})\to 0$ due to  the existence of unbounded many $t>0$. The same reason also gives $$1<\vert q_0\vert_p\vert q_0\vert_\infty\leq \max (p^{-\gamma t}, \Vert\bq\Vert_p\vert q_0\vert_\infty\Vert\by\Vert_p)$$
		 which implies that $$1\leq \Vert \bq\Vert_p\vert q_0\vert_{\infty}\Vert\by\Vert_p $$ if $q_0$ is nonzero, which says that $\vert q'_0\vert_{+}\geq c$ where $c>0$ depends only on $\by$. We denote $q_i=p^{l_i}z_i\in\mathcal{D}$ where $z_i$ are integers without any $p$ factor. Then $$\vert q'_i\vert_{+}= \left\{\begin{array}{ll}
		\vert q_i\vert_{\infty}\Vert\bq\Vert_p &\text{ if } q_i\neq 0\\
		\\
		1 &\text{otherwise}. 
		\end{array}\right.$$ Suppose $\Pi_{+}(\tilde{\bq}')\leq M$ for some $M>0$, i.e. $\prod_{i=0,\cdots,n} (\vert q'_i\vert_{+})\leq M$ which in turn implies that
		$$ \vert q'_0\vert_{+}\prod_{\{i>1~|~q_i\neq 0\}}\vert z_i\vert_\infty\leq M.$$ 
		Here we are using that if $\Vert \bq\Vert_p=p^l$ then $-l_i\leq l$ for $i>1$. Since $\vert q'_0\vert_{+}$ is bounded below by positive, there are finitely many options for the integers $z_i$ occuring in $q_i$ for $i>1$. Hence we have that 
		$$
		\vert q'_0\vert_{+}\prod_{\{i>1~|~q_i\neq 0\}}(p^{l_i}\Vert\bq\Vert_p)\leq M'
		$$ for some $M'>0$.
	 The above inequality gives \begin{equation}\label{eqn1}
		\vert q'_0\vert_{+}\prod_{\{i>1~|~q_i\neq 0\}} p^{l+l_i}\leq M'.\end{equation}
		 Hence we have $\vert q'_0\vert_{+}\leq M'$ and arguing as in (\ref{condition}) we may conclude that 
		 $$\vert z_0\vert_\infty=\vert q'_0\vert_{\infty}\vert q'_0\vert_p\leq \max (1, M'\Vert \by\Vert_p),$$ 
		 has only finitely many options. Now from (\ref{eqn1}) $$
		 c\leq\prod_{\{i~|~q_i\neq 0\}} p^{l+l_i} \leq M'.$$\noindent
		Thus we now have $-\alpha\leq\sum_{\{i~|~q_i\neq 0\}} l_i+ml\leq \alpha$ for some $\alpha>0$. Since $c\leq \vert q_0\vert _{+}\leq M'$ which gives $-\beta\leq (l_0+l)\leq \beta$  for some $\beta>0$. Hence for $i>1$ and $q_i\neq 0$ we have $0\leq l_i+l\leq \alpha'$ for some $\alpha'>0$. But then only way $ \vert q_0+\bq\cdot \by\vert_p^m\cdot \Pi_{+}(\tilde{\bq})$ can go to $0$ is if there exists some $\tilde\bq(\neq 0 )\in \Z^{n+1}$ such that $q_0+\bq\cdot\by=0$.
		\noindent
		 In that case $\by$ is very well multiplicatively approximable. So if $\by$ is not such then $\Pi_{+}(\tilde{\bq}')$ has to be unbounded and satisfies 
		$$
		\vert q'_0+\bq'\cdot\by\vert_p\leq \Pi_{+}(\tilde{\bq}')^{-(1+\varepsilon)}
		$$
		where $\tilde{\bq}'=(q'_0,\bq')\in\Z[1/p]\times\Z^n$. Another crucial observation now is that $\vert q'_0\vert_p$ is bounded above by a constant depending on $\by$. So in case $q'_0\notin \Z$ we can write $$
		\big\vert \vert q'_0\vert_p\cdot (q_0'+\bq'\cdot\by)\big\vert_p\vert q'_0\vert_p\leq \Pi_{+}(\tilde{\bq}')^{-(1+\varepsilon)}.
		$$ Now taking $\tilde {\bq}''=\vert q_0\vert_p\tilde{\bq}'$ and using the upper bound on $\vert q_0\vert_p$ enables us to conclude 
		$$  
		\vert q''_0+\bq''\cdot\by\vert_p\leq \Pi_{+}(\tilde{\bq}'')^{-(1+\varepsilon')}
		$$ for infinitely many $\tilde{\bq}''\in \Z^{n+1}$ with $\varepsilon'<\varepsilon$. Therefore $\by$ is very well multiplicatively approximable.\\
		We now prove the converse. Suppose we have that \begin{equation}\label{vwma}
		\vert q_0+\bq\cdot\by\vert_p\leq \Pi_{+}(\tilde{\bq})^{-(1+\varepsilon)}
		\end{equation} for infinitely many $\tilde{\bq}\in \Z^{n+1}$.
		Then choose $t_i>0$ such that $\vert q_i\vert_{+}=\Pi_{+}(\tilde{\bq})^{-\frac{\varepsilon}{n+1}} p^{t_i}$. Multiplying these we get $p^t=\Pi_{+}(\tilde{\bq})^{(1+\varepsilon)}$ which guarantees $t$ to be unbounded. The choice of $t_i$ gives the following condition $$
		p^{-t_i}\vert q_i\vert_\infty\leq  p^{-t_i}\vert q_i\vert_{+}\leq \Pi_{+}(\tilde{\bq})^{-\frac{\varepsilon}{n+1}}=p^{-\gamma t} ~~\forall ~i=0,\cdots,n$$ where $\gamma=\frac{\varepsilon}{(n+1)(1+\varepsilon)}$. Hence $\Vert g_{\bt}^\infty u_{\by}^\infty\tilde{\bq}\Vert_\infty\leq p^{-\gamma t}.$ On the other hand, 
		$$\vert p^{-t}(q_0+\bq\cdot\by)\vert_p=p^t\vert q_0+\bq\cdot\by\vert_p\leq p^t\cdot\frac{1}{p^t}=1,$$ due to (\ref{vwma}). Therefore we have $c(g_{\bt}u_{\by}\tilde{\bq})\leq p^{-\gamma t}$. Now taking $[\bt]$ consisting of integer factors and observing that the ratio of $\delta (g_{[\bt]}u_{\by}\mathcal{D}^{n+1})$ and $\delta (g_{\bt}u_{\by}\mathcal{D}^{n+1})$ is bounded by uniform factor. Hence reducing $\gamma$ a bit we can conclude the lemma.
		
	\end{proof}
	
	Note that one direction $(\Longrightarrow)$ of the last lemma was already observed by Kleinbock and Tomanov in \cite{KT} with a slight variation. The main content of the previous lemma is that we can come back from dynamics to number theory using the reverse direction, which was earlier not known to the best of our  knowledge.
	
	\begin{theorem}
		For any analytic manifold of $\Q_p^n$ if one point is not very well multiplicatively approximable then almost every point in a neighbourhood of that point is not very well multiplicatively approximable.
	\end{theorem}
	\begin{proof}
		Following the proof of Theorem \ref{gamma} and  considering $$U_1:=\left\{\bx\in U~\Bigg | ~\begin{aligned}&\text{ for any } 0<\gamma<\frac{1}{n+1} ,~ \delta(g_{\bt}u_{\by}\mathcal{D}^{n+1})> p^{-\gamma t} \\
		&\text{ for all but finitely many } \bt\in\Z_+^{n+1}\end{aligned}\right\}$$ one can show that $U_1$ has positive measure if nonempty. Then the using the above Lemma \ref{dani_vwma} one can conclude that if one point is not very well multiplicatively approximable then $U_1\neq\emptyset$, hence of positive measure. And then again using Lemma \ref{dani_vwma}, the conclusion follows.
	\end{proof}

	We now turn to the proof of Theorem \ref{thm:mult}. Note that by repeating the argument as in Proposition $6.1$ of \cite{DGpart1}, it can be proved that: 
	\begin{proposition}\label{prop_m1}
		Take $\mathcal{R}=\Q_p\times \R$. Let $X$ be a Besicovitch metric space and $\mu$ be a uniformly Federer measure on $X$. Denote $\tilde B:=B(x,3^{n+1}r)$. Suppose we are given a continuous function $\f : X\mapsto \Q_{p}^{n}$ and $C,\alpha >0$ with the following properties\\
		{\label{b11}\rm{(i)}}~$x\mapsto \cov(g_{\bt} u_{\f(x)}\Delta) $is $(C,\alpha)$ good  with respect to  $\mu\text{ in }\tilde B  ~\forall ~\Delta \in \mathcal P(\mathcal D,n+1),$\\
		{\label{b123}\rm{(ii)}}  for any $d>0 $ there exists $T=T(d) >0$ such that for any $\bt\in\Z_+^{n+1}$ with  $t\geq T$ and any $\Delta \in \mathcal P(\mathcal D,n+1)$ one has 
		\begin{equation}
		\sup_{x\in B\cap \supp\mu}\cov (g_{\bt}u_{\f(x)}\Delta)\geq p^{-(\rk\Delta)dt}.
		\end{equation}
		Then $\mu$ a.e every $x$, $\f(x)$ is not VWMA.
		
	\end{proposition}
Note that condition $\rm{(ii)}$ in the above Lemma is actually necessary. If $\rm{(ii)}$ does not hold then there exists unbounded $t$ with $\bt\in\Z_+^{n+1} $ such that $$\delta(g_{\bt}u_{\f(x)}\mathcal{D}^{n+1})\leq p^{-dt}$$ for some $0<d<\frac{1}{n+1}$ for all $x\in B\cap\supp\mu$. Now from the above lemma we have that $\f(x)\in\mathcal{WM}_{\frac{md+m-kd}{m(1-kd)}}$ for some $1\leq m,k\leq n+1$. Since $\frac{md+m-kd}{m(1-kd)}>\frac{1}{1-d}>1$ we have that $\f(B\cap\supp\mu)\subset \mathcal{WM}_{\frac{1}{1-d}}$ i.e $\f(x)$ is VWMA for all $x\in B\cap\supp\mu$. 
\noindent Also condition $\rm{(ii)}$ in the above proposition is same as \begin{equation}\label{covolume1m}\begin{aligned}
& \text{ for any } d>0~
\exists ~T=T(d) >0 \text{ such that for any } \bt\in\Z_+^{n+1} \text{ with } t\geq T \\& \forall~j=1,\cdots, n \text{ and } \forall~\bw\in \bigwedge^{j}\mathcal{D}^{n+1},\text{ one has } \\&
\max\bigg( p^{t}\Vert R\bc(\bw) \Vert_{p},\Vert \pi(\bw)\Vert_{p}\bigg)\max_I{p^{-t_{I}}\Vert w_I\Vert_\infty}\geq p^{-jdt},
\end{aligned}
\end{equation} 
which is independent of $\f$ but depends on the affine subspace in which manifold is nondegenerate.

\noindent Combining all these previous observations and repeating the same arguments as in section $\S 7$ of \cite{DGpart1}, we have the following. 
\begin{theorem}
	Let $\mu$ be a Federer measure on a Besicovitch metric space $X, \mathcal{L}$ an affine subspace of $\Q_p^n$, and let $\f : X\to \mathcal{L}$ be a continuous map such that $(\f, \mu)$ is good and nonplanar. Then the following are equivalent: \\
	\rm{(i)} \begin{equation}
	\big\{x~\in\supp\mu ~| \f(x) \text{ is not VWMA }\big\} \text{ is nonempty}.
	\end{equation}
	\noindent\rm{(ii)}\begin{equation} ~~~ 
\f(x) \text{ is not VWMA for } \mu\text{ a.e } x.
	\end{equation}
	\rm{(iii)}\begin{equation}
\text{ Condition } (\ref{covolume1m}) \text{ holds .}
	\end{equation}
\end{theorem}
Theorem \ref{thm:mult} now follows as a corollary using Theorem $4.3$ of \cite{KT}.



\begin{thebibliography}{99}
\bibitem{BBDD} V. Beresnevich, V. Bernik, H. Dickinson and M. M. Dodson, \textit{On linear manifolds for which the Khintchin approximation theorem holds}, Vestsi Acad Navuk Belarusi. Ser. Fiz. - Mat. Navuk (2000), 14--17 (Belorussian).
\bibitem{BGGV}  V. Beresnevich, A. Ganguly, A. Ghosh and S. Velani, \textit{Inhomogeneous dual Diophantine approximation on affine subspaces}, https://arxiv.org/abs/1711.08559, to appear in International Mathematics Research Notices.
\bibitem{BKsurvey} V. Beresnevich and D. Kleinbock, \textit{Quantitative non-divergence and Diophantine approximation on manifolds}, https://arxiv.org/abs/1906.00747.
\bibitem{BK} V.V. Beresnevich, E.I. Kovalevskaya, \textit{On Diophantine approximations of dependent quantities in the p-adic case}, Mat. Zametki 73:1 (2003), 22--37; translation: Math. Notes 73:1-2 (2003), 21--35.
\bibitem{Bugeaud} Y. Bugeaud, \textit{Multiplicative Diophantine approximation}, Dynamical systems and Diophantine approximation, 105--125, 
Semin. Congr., 19, Soc. Math. France, Paris, 2009. 
\bibitem{CaoYu} R. Cao and J. You, \textit{Diophantine vectors in analytic submanifolds of Euclidean spaces}, Sci. China Ser. A
50 (2007) 1334--8.
\bibitem{DG} Shreyasi Datta and Anish Ghosh, \textit{S-arithmetic Inhomogeneous Diophantine approximation on manifolds}, arXiv:1801.08848v1.
\bibitem{DGpart1} Shreyasi Datta and Anish Ghosh, \textit{Diophantine inheritance for $p$-adic measures}, https://arxiv.org/abs/1903.09362.
\bibitem{GG} A. Ganguly and A. Ghosh, \textit{Quantitative Diophantine approximation on affine subspaces}, https://arxiv.org/abs/1610.02157, to appear in Math. Z.
\bibitem{G1} A. Ghosh, \textit{A Khintchine-type theorem for hyperplanes}, J. London Math. Soc. \textbf{72}, No.2 (2005), pp. 293--304.
\bibitem{G-thesis} A. Ghosh, \textit{ Dynamics of homogeneous spaces and Diophantine approximation on manifolds}, Thesis (Ph.D.) Brandeis University. 2006. 55 pp. ISBN: 978-0542-56323-2.
\bibitem{G-div} A. Ghosh, \textit{A Khintchine Groshev Theorem for affine hyperplanes}, International Journal of Number Theory, 7 (2011), no. 4, 1045--1064.
\bibitem{G-mult} A. Ghosh, \textit{Diophantine approximation on affine hyperplanes}, Acta Arithmetica, 144 (2010), 167--182.
\bibitem{G-monat} A. Ghosh, \textit{Diophantine approximation and the Khintchine-Groshev theorem}, Monatsh. Math. \textbf{163} (2011), no. 3, 281--299.
\bibitem{G-handbook} A. Ghosh, \textit{Diophantine approximation on subspaces of $\mathbb{R}^n$ and dynamics on homogeneous spaces},  Handbook of Group Actions (Vol. IV) ALM 41, Ch. 9, pp. 509--527.
\bibitem{GM} A. Ghosh and A. Marnat, \textit{On Diophantine transference principles},  Mathematical Proceedings of the Cambridge Philosophical Society Volume 166, Issue 3 May 2019 , pp. 415--431.
\bibitem{Kleinbock-extremal} D. Kleinbock, \textit{Extremal subspaces and their submanifolds}, Geom. Funct. Anal \textbf{13}, (2003), No 2, pp.437--466.
\bibitem{Kleinbock-exponent} D. Kleinbock, \textit{An extension of quantitative nondivergence and applications to Diophantine exponents}, Trans. Amer. Math. Soc. 360 (2008), no. 12, 6497--6523.
\bibitem{Kleinbock-dichotomy} D. Kleinbock, \textit{An `almost all versus no' dichotomy in homogeneous dynamics and Diophantine approximation}, Geom. Dedicata 149 (2010), 205--218.
\bibitem{KLW} D. Kleinbock, E. Lindenstrauss and B. Weiss, \textit{On fractal measures and Diophantine approximation}, 
Selecta Math. (N.S.) 10 (2004), no. 4, 479--523.
\bibitem{KM} D. Kleinbock and G. A. Margulis, \textit{Flows on homogeneous spaces and Diophantine Approximation on Manifolds}, Ann Math\textbf{148}, (1998), pp.339--360.
\bibitem{KT} D. Kleinbock and G. Tomanov, \textit{Flows on $S$-arithmetic homogeneous spaces and applications to metric Diophantine approximation}, Comm. Math. Helv. 82 (2007), 519--581.
\bibitem{MoS1} A. Mohammadi, A. Salehi Golsefidy, \textit{$S$-arithmetic Khintchine-type theorem}, Geom. Funct. Anal. 19 (2009), no. 4, 1147--1170.
\bibitem{MoS2} A. Mohammadi, A. Salehi Golsefidy, \textit{Simultaneous Diophantine approximation on non-degenerate p-adic manifolds}, Israel J. Math. 188 (2012), 231--258. 
\bibitem{Mosh} N. Moshchevitin, \textit{On Kleinbock's Diophantine result}
Publ. Math. Debrecen 79 (2011), no. 3-4, 531--537.
\bibitem{Zhang1}  Y. Zhang, \textit{Multiplicative Diophantine exponents of hyperplanes and their non-degenerate submanifolds}, Journal f\"{u}r die reine und angewandte Mathematik 664 (2012), 93--113.
		
		
	\end{thebibliography}
\end{document}